\documentclass[11pt]{amsart}
\usepackage{amsmath}
\usepackage{amscd}
\usepackage[utf8]{inputenc}
\usepackage{mathrsfs} 
\usepackage{amssymb}
\usepackage{bm}
\usepackage{graphicx}
\usepackage[centertags]{amsmath}
\usepackage{amsfonts}
\usepackage{amsthm}
\usepackage{enumerate}
\usepackage{tocvsec2}
\usepackage{xcolor}
\usepackage[margin=1in]{geometry}

\newtheorem{theorem}{Theorem}[section]
\newtheorem*{theorem*}{Theorem}

\newtheorem{corollary}[theorem]{Corollary}
\newtheorem{lemma}[theorem]{Lemma}
\newtheorem{rem}[theorem]{Remark}

\newtheorem{proposition}[theorem]{Proposition}

\theoremstyle{definition}

\newcommand{\ee}{\varepsilon}
\newcommand{\nn}{\mathbb{N}}
\newcommand{\rr}{\mathbb{R}}

\begin{document}
\title{Concerning summable Szlenk index}
\author{R.M. Causey}

\begin{abstract} We generalize the notion of summable Szlenk index from a Banach space to an arbitrary weak$^*$-compact set.  We prove that a weak$^*$-compact set has summable Szlenk index if and only if its weak$^*$-closed, absolutely convex hull does. As a consequence, we offer a new, short proof of a result from \cite{DK2} regarding the behavior of summability of Szlenk index under $c_0$ direct sums.  We also use this result to prove that the injective tensor product of two Banach spaces has summable Szlenk index if both spaces do, which answers a question from \cite{DK}. As a final consequence of this result, we prove that a separable Banach space has summable Szlenk index if and only if it embeds into a Banach space with an asymptotic $c_0$ finite dimensional decomposition, which generalizes a result from \cite{OSZ}. We also introduce  an ideal norm $\mathfrak{s}$ on the class $\mathfrak{S}$ of operators with summable Szlenk index and prove that $(\mathfrak{S}, \mathfrak{s})$ is a Banach ideal. For $1\leqslant p\leqslant \infty$, we prove precise results regarding the summability of the Szlenk index of an $\ell_p$ direct sum of a collection of operators.

\end{abstract}
\maketitle

\section{Introduction}

Since its inception in \cite{Sz}, the Szlenk index has been a fundamental object in the geometry of Banach space theory, including the non-linear theory (see \cite{GKL2} and \cite{GKL}).  The Szlenk index and Szlenk power type are fundamentally connected to asymptotically uniformly smooth renorming properties of spaces and operators, as was shown in \cite{C2}, \cite{CD}, \cite{KOS}, \cite{R}. Such properties have seen significant recent use in the non-linear asymptotic theory (see \cite{BKL}, \cite{KR}, \cite{LR}).  Of particular importance is the notion of a Banach space having summable Szlenk index. In \cite{GKL}, a characterization is given of those separable Banach spaces which have summable Szlenk index in terms of the behavior of the modulus of asymptotic uniform smoothness under equivalent norms. Furthermore, it is shown there that if $X$ has summable Szlenk index, and if $Y$ is uniformly homeomorphic to $X$, then $Y$ has summable Szlenk index.

In this work, we define what it means for a weak$^*$-compact set to have summable Szlenk index, which generalizes the notion of a Banach space having summable Szlenk index.  Our first result is the following.

\begin{theorem} Let $X$ be a Banach space and let $K\subset X^*$ be weak$^*$-compact.  Then $K$ has summable Szlenk index if and only if $\overline{\text{\emph{abs\ co}}}^{\text{\emph{weak}}^*}(K)$ does. 

\label{mor}

\end{theorem}

Our first  application of this result is the following embedding result, which generalizes a result from \cite{OSZ}.

\begin{theorem} If $X$ is a separable Banach space, then $X$ has summable Szlenk index if and only if there exists a Banach space $Z$ with FDD $E$ such that $E$ is asymptotic $c_0$ in $Z$ and $Z$ admits a subspace isometric to $X$. 

\label{emb}
\end{theorem}

Our second application answers a question posed in \cite{DK}. 

\begin{theorem} Let $A_0:X_0\to Y_0$, $A_1:X_1\to Y_1$ be bounded, linear operators. If $A_0, A_1$ have summable Szlenk index, then so does the induced operator $A_0\otimes A_1: X_0\hat{\otimes}_\ee X_1\to Y_0\hat{\otimes}_\ee Y_1$ between the injective tensor products.  If neither $A_0$ nor $A_1$ is the zero operator, then the converse holds as well.

\end{theorem}

One last application of Theorem \ref{mor} is a short proof of an operator version of a result from \cite{DK2}.

\begin{theorem} Let $\Lambda$ be a non-empty set and let $A_\lambda:X_\lambda\to Y_\lambda$ be a uniformly bounded collection of linear operators. Then the induced operator $A:(\oplus_{\lambda\in \Lambda} X_\lambda)_{c_0(\Lambda)}\to (\oplus_{\lambda\in \Lambda}Y_\lambda)_{c_0(\Lambda)}$ has summable Szlenk index if and only if the operators $A_\lambda$ have uniformly summable Szlenk index.

\end{theorem}

We also study the ideal properties of the class $\mathfrak{S}$ of operators with summable Szlenk index, as well as introduce a way to assign to each operator $A$ a value $\Sigma(A)\in [0, \infty]$ such that $A$ has summable Szlenk index if and only if $\Sigma(A)<\infty$.  Moreover, the quantity $\mathfrak{s}(A):=\|A\|+ \Sigma(A)$ defines an  ideal norm  on $\mathfrak{S}$.  In this direction, we prove the following. 

\begin{theorem} The class $(\mathfrak{S}, \mathfrak{s})$ is a Banach ideal.

\end{theorem}

We also study the behavior of summable Szlenk index of $\ell_p$ direct sums of operators for $1\leqslant p\leqslant \infty$. Such a study is trivial in the setting of spaces, since the norm of the identity operator of a Banach space is either $0$ or $1$, but non-trivial for operators.  We prove the following. 

\begin{theorem} Let $\Lambda$ be a non-empty set and let $A_\lambda:X_\lambda\to Y_\lambda$ be a uniformly bounded collection of linear operators. Then for any $1\leqslant p\leqslant \infty$, the induced operator $A:(\oplus_{\lambda\in \Lambda} X_\lambda)_{\lambda_p(\Lambda)}\to (\oplus_{\lambda\in \Lambda} Y_\lambda)_{\ell_p(\Lambda)}$ has summable Szlenk index if and only if $(\|A_\lambda\|)_{\lambda\in \Lambda}\in c_0(\Lambda)$ and $(\Sigma(A_\lambda))_{\lambda\in \Lambda}\in \ell_p(\Lambda)$.

\end{theorem}

\section{Definitions}

Throughout, $\mathbb{K}$ will denote the scalar field (either $\mathbb{R}$ or $\mathbb{C}$), and ``operator'' will mean ``bounded, linear operator.''  

Given a directed set $D$ and $n\in \nn$, we let $D^{\leqslant n}=\cup_{i=1}^n D^i$.  Given $t=(u_i)_{i=1}^k \in D^{\leqslant n}$, we let $|t|=k$, $t^-=(u_i)_{i=1}^{k-1}$ (where $(u_i)_{i=1}^0=\varnothing$ by convention), and for $0\leqslant j\leqslant k$, $t|_j=(u_i)_{i=1}^j$.    Given $t=(u_i)_{i=1}^k\in \{\varnothing\}\cup D^{\leqslant n-1}$ and $u\in D$, we let $t\smallfrown u=(u_1, \ldots, u_k, u)\in D^{\leqslant n}$.  For $s,t\in D^{\leqslant n}$, we let $s\smallfrown t$ be the concatenation of $s$ and $t$.   If $X$ is a Banach space, we say a collection $(x_t)_{t\in D^{\leqslant n}}\subset X$ is \emph{weakly null} provided that for every $t\in \{\varnothing\}\cup D^{\leqslant n-1}$, $(x_{t\smallfrown u})_{u\in D}$ is a weakly null net. We say a map $\phi:D^{\leqslant n}\to D^{\leqslant n}$ is a \emph{pruning} provided that   $|\phi(t)|=|t|$ and $\phi(t)^-=\phi(t^-)$ for each $t\in D^{\leqslant n}$ and such that the collection $(x'_t)_{t\in D^{\leqslant n}}$ is weakly null, where $x'_t=x_{\phi(t)}$. The following can be easily proved by induction on $n$.  We will use this result frequently. 

\begin{proposition} Let $D$ be a directed set, $n\in \nn$, $X$ a Banach space, and $(x_t)_{t\in D^{\leqslant n}}$ a weakly null collection.  Let $(M, d)$ be a compact metric space and suppose $F:D^n\to M$ is any function. Then for any $\delta>0$, there exist a pruning $\phi:D^{\leqslant n}\to D^{\leqslant n}$ and $\varpi\in M$ such that $d(\varpi, \phi(t))\leqslant \delta$ for any $t\in D^n$.  

\label{tree}

\end{proposition}

For a Banach space $X$ and $n\in \nn$, we let $\{X\}_n$ denote the set of all norms $|\cdot|$ on $\mathbb{K}^n$ such that for any $b>1$, there exists a weakly null collection $(x_t)_{t\in D^{\leqslant n}}\subset S_X$ such that for any $(a_i)_{i=1}^n\in S_{\ell_\infty^n}$ and any $t\in D^n$, $$b^{-1} \|\sum_{i=1}^n a_i x_{t|_i}\|\leqslant |\sum_{i=1}^n a_i e_i| \leqslant b\|\sum_{i=1}^n a_i x_{t|_i}\|.$$  A standard compactness argument yields that $\{X\}_n\neq \varnothing$ whenever $\dim X=\infty$.  In keeping with the terminology in \cite{JKO},  we say that $X$ is \emph{Asymptotic } $c_0$ if $\dim X=\infty$ and there exists a constant $C\geqslant 1$ such that $$C^{-1}\|\sum_{i=1}^n a_i e_i\|_{\ell_\infty^n}\leqslant |\sum_{i=1}^n a_i e_i|\leqslant C\|\sum_{i=1}^n a_i e_i\|_{\ell_\infty^n}$$ for each $n\in \nn$, each $|\cdot|\in \{X\}_n$, and each $(a_i)_{i=1}^n\in \mathbb{K}^n$.  We remark that since the canonical $\mathbb{K}^n$ basis is normalized and monotone in $(\mathbb{K}^n, |\cdot|)$ for each $|\cdot|\in \{X\}_n$, we always have $$\|\sum_{i=1}^n a_i e_i\|_{\ell_\infty^n}\leqslant 2|\sum_{i=1}^n a_i e_i|,$$ so the upper inequality is the only one we need to check in order to establish that a given infinite dimensional space is Asymptotic $c_0$.

We recall that a sequence $E=(E_n)_{n=1}^\infty$ of finite dimensional subspaces of the Banach space $X$ is called a \emph{finite dimensional decomposition} (resp. \emph{FDD}) for $X$ provided that for each $x\in X$, there exists a unique sequence $(x_n)_{n=1}^\infty \in \prod_{n=1}^\infty E_n$ such that $x=\sum_{n=1}^\infty x_n$.  In this case, for each $m\in \nn$, the projection $P^E_m\sum_{n=1}^\infty x_n=\sum_{n=1}^m x_n$ is continuous. We let $P^E_0=0$. By the Principle of Uniform Boundedness, $\sup_{0\leqslant m< n}\|P^E_n-P^E_m\|<\infty$.   If $0=m_0<m_1<\ldots$ and $F_n=\oplus_{i=m_{n-1}+1}^{m_n} E_i$, then $F=(F_n)_{n=1}^\infty$ is also an FDD for $X$. In this case, we say $F$ is a \emph{blocking} of $E$.  We say $E$ is \emph{shrinking} if $\{(P^E_n)^*(X^*): n\in \nn\}$ is dense in $X^*$, which occurs if and only if $(E^*_n)_{n=1}^\infty$ is   an FDD for $X^*$. Here,  $E^*_n$ is identified with $(P^E_n)^*(X^*)\cap \ker P^E_{n-1}$.  We say $E$ is \emph{asymptotic} $c_0$ in $X$ if there exists $C\geqslant 1$ such that for any $n\leqslant k_0<\ldots <k_n$ and any $x_i\in  \oplus_{j=k_{i-1}}^{k_i-1} E_j$, $$ C^{-1}\max_{1\leqslant i\leqslant n}\|x_i\|\leqslant \|\sum_{i=1}^n x_i\|\leqslant C\max_{1\leqslant i\leqslant n}\|x_i\|.$$   We remark that if $b=\sup_{0 \leqslant m< n} \|P^E_n-P^E_m\|$ and if $(x_i)_{i=1}^n$ is any block sequence with respect to $E$, $$\max_{1\leqslant i\leqslant n}\|x_i\|\leqslant b\|\sum_{i=1}^n x_i\|,$$ so the upper inequality is the only one we need to check in order to establish that $E$ is asymptotic $c_0$ in $X$.

Given a Banach space $X$, a weak$^*$-compact subset $K$ of $X^*$, and $\ee>0$, we let $s_\ee(K)$ denote the subset of $K$ consisting of those $x^*\in K$ such that for each weak$^*$-neighborhood $V$ of $K$, $\text{diam}(V\cap K)>\ee$. For convenience, we let $s_\ee(K)=K$ whenever $\ee\leqslant 0$. We then define by transfinite induction $$s^0_\ee(K)=K,$$ $$s^{\xi+1}_\ee(K)=s_\ee(s^\xi_\ee(K)),$$ and if $\xi$ is a limit ordinal, we let $$s^\xi_\ee(K)=\bigcap_{\zeta<\xi}s^\zeta_\ee(K).$$  If there exists an ordinal $\xi$ such that $s^\xi_\ee(K)=\varnothing$, we let $Sz(K, \ee)$ be the minimum such ordinal.   If no such ordinal exists, we write $Sz(K, \ee)=\infty$.  We define $Sz(K)=\sup_{\ee>0}Sz(K, \ee)$.   If $A:X\to Y$ is an operator, we let $Sz(A, \ee)=Sz(A^*B_{Y^*}, \ee)$ and $Sz(A)=Sz(A^*B_{Y^*})$.  If $X$ is a Banach space, we let $Sz(X, \ee)=Sz(B_{X^*}, \ee)$ and $Sz(X)=Sz(B_{X^*})$. For $M\geqslant 0$, we say $K$ has $M$-\emph{summable Szlenk index} provided that if $\ee_1, \ldots, \ee_n\in \rr$  (equivalently, if $\ee_1, \ldots, \ee_n>0$) are such that $s_{\ee_1}\ldots s_{\ee_n}(K)\neq \varnothing$, $\sum_{i=1}^n \ee_i\leqslant M$.  This implies that $Sz(K, \ee)\leqslant M/\ee+1$ for all $\ee>0$, and in particular, $Sz(K)\leqslant \omega$.  We say $K$ has \emph{summable Szlenk index} if it has $M$-summable Szlenk index for some $M\geqslant 0$.

We let $\textbf{Ban}$ denote the class of all Banach spaces over $\mathbb{K}$. We let $\mathfrak{L}$ denote the class of all operators between Banach spaces and for $X,Y\in \textbf{Ban}$, we let $\mathfrak{L}(X,Y)$ denote the set of operators from $X$ into $Y$. For $\mathfrak{I}\subset \mathfrak{L}$ and $X,Y\in \textbf{Ban}$, we let $\mathfrak{I}(X,Y)=\mathfrak{I}\cap \mathfrak{L}(X,Y)$. We recall that a class $\mathfrak{I}$ is called an ideal if \begin{enumerate}[(i)]\item For any $W,X,Y,Z\in \textbf{Ban}$, any $C\in \mathfrak{L}(W,X)$, $B\in \mathfrak{I}(X,Y)$, and $A\in \mathfrak{L}(Y,Z)$, $ABC\in \mathfrak{I}$, \item $I_\mathbb{K}\in \mathfrak{I}$, \item for each $X,Y\in \textbf{Ban}$, $\mathfrak{I}(X,Y)$ is a vector subspace of $\mathfrak{L}(X,Y)$. \end{enumerate}

We recall that an ideal $\mathfrak{I}$ is said to be  \emph{closed} provided that for any $X,Y\in \textbf{Ban}$, $\mathfrak{I}(X,Y)$ is closed in $\mathfrak{L}(X,Y)$ with its norm topology. 

If $\mathfrak{I}$ is an ideal and $\iota$ assigns to each member of $\mathfrak{I}$ a non-negative real value, then we say $\iota$ is an \emph{ideal norm} provided that \begin{enumerate}[(i)]\item for each $X,Y\in \textbf{Ban}$, $\iota$ is a norm on $\mathfrak{I}(X,Y)$, \item for any $W,X,Y,Z\in \textbf{Ban}$ and any $C\in \mathfrak{L}(W,X)$, $B\in \mathfrak{I}(X,Y)$, $A\in \mathfrak{I}(Y,Z)$, $\iota(ABC)\leqslant \|A\|\iota(B)\|C\|$, \item for any $X,Y\in \textbf{Ban}$, any $x\in X$, and any $y\in Y$, $\iota(x\otimes y)=\|x\|\|y\|$. \end{enumerate}

If $\mathfrak{I}$ is an ideal and $\iota$ is an ideal norm on $\mathfrak{I}$, we say $(\mathfrak{I}, \iota)$ is a \emph{Banach ideal} provided that for every $X,Y\in \textbf{Ban}$, $(\mathfrak{I}(X,Y), \iota)$ is a Banach space.

\section{An ideal seminorm}

Given a Banach space $X$ and a weak$^*$-compact subset $K$ of $X^*$, for $x\in X$,  we let $r_K(x)=0$ if $K=\varnothing$, and otherwise we let $r_K(x)=\max_{x^*\in K} \text{Re\ }x^*(x)$. We note that $r_{A^*B_{Y^*}}(x)=\|Ax\|$,  $r_{B_{X^*}}(x)=\|x\|$, and $r_K=r_{\overline{\text{co}}^{\text{weak}^*}(K)}$ for any weak$^*$-compact $K$.   Note also that $r_K$ is a sublinear functional, and it is a seminorm if $K$ is balanced.   Given $n\in \nn$, we let $\Sigma_n(K)$ be the infimum of those $s>0$ such that for every directed set $D$, every weakly null $(x_t)_{t\in D^{\leqslant n}}\subset B_X$, $$\inf_{t\in D^n} r_K\bigl(\sum_{i=1}^n x_{t|_i}\bigr) \leqslant s.$$  We let $\Sigma(K)=\sup_n \Sigma_n(K)$.  If $A:X\to Y$ is an operator, we let $\Sigma_n(A)=\Sigma_n(A^*B_{Y^*})$, $\Sigma(A)=\Sigma(A^*B_{Y^*})$.  If $X$ is a Banach space, we let $\Sigma_n(X)=\Sigma_n(I_X)$, $\Sigma(X)=\Sigma(I_X)$. 

\begin{rem}\upshape We note that it is convenient to allow any directed set in the definition of $\Sigma_n$. However, we obtain the same value of $\Sigma_n(K)$ if in the definition we only consider weakly null collections indexed by $D^{\leqslant n}_1$, where $D_1$ is a fixed weak neighborhood basis at $0$ in $X$. Indeed, if $(x_t)_{t\in D^{\leqslant n}}\subset B_X$ is such that $\inf_{t\in D^n} r_K(\sum_{i=1}^n x_{t|_i})$, one can define by induction some map $\phi:D_1^{\leqslant n}\to D$ such that $|\phi(t)|=|t|$, $\phi(t^-)=\phi(t)^-$ for $t\in D^{\leqslant n}$ and $(x_{\phi(t)})_{t\in D_1^{\leqslant n}}$ is also weakly null.   

\label{remrem}
\end{rem}

In what follows, $\mathbb{S}$ denotes the set of unimodular scalars. We let $\mathbb{S}K=\{\ee x^*: \ee \in \mathbb{S}, x^*\in K\}$. 

\begin{proposition} Let $X$ be a Banach space, $K\subset X^*$ weak$^*$-compact, and $n\in \nn$.  \begin{enumerate}[(i)]\item $K$ is norm compact if and only if $\Sigma_1(K)=0$ if and only if $\Sigma(K)=0$. \item If $R\geqslant 0$ is such that $K\subset RB_{X^*}$, $\Sigma_n(K)\leqslant Rn$. \item If $L\subset X^*$ is weak$^*$-compact, then $\Sigma_n(K+L)\leqslant \Sigma_n(K)+\Sigma_n(L)$. \item If $\ee$ is a unimodular scalar, $\Sigma_n(\ee K)=\Sigma_n(K)$. \item If $K_1, \ldots, K_l\subset X^*$ are weak$^*$-compact, $\Sigma_n\bigl(\cup_{i=1}^l K_i\bigr)=\max_{1\leqslant i\leqslant l} \Sigma_n(K_i).$  \item  $\Sigma_n(K)=\Sigma_n(\mathbb{S}K)$. \item $\Sigma_n(\overline{\emph{\text{abs\ co}}}^{\text{\emph{weak}}^*}(K))=\Sigma_n(K)$.  \item For $s>0$, $\Sigma_n(K)<s$ if and only if for every $(x_t)_{t\in D^{\leqslant n}}\subset B_X$, there exists a pruning $\phi:D^{\leqslant n}\to D^{\leqslant n}$ such that $$\sup_{t\in D^n} \sup_{(a_i)_{i=1}^n\in B_{\ell_\infty^n}} r_K(\sum_{i=1}^n a_i x_{\phi(t)|_i})<s.$$  \item If $\dim X=\infty$, $\Sigma_n(K)$ is the infimum of those $s>0$ such that for every directed set $D$ and every weakly null $(x_t)_{t\in D^{\leqslant n}}\subset S_X$, $$\inf_{t\in D^n}r_K(\sum_{i=1}^n x_{t|_i})\leqslant s.$$   \end{enumerate}

\label{tech1}
\end{proposition}

\begin{proof}$(i)$ Since $r_K$ is a sublinear functional, it follows that $\Sigma_n(K)\leqslant n \Sigma_1(K)$, so $\Sigma(K)=0$ if and only if $\Sigma_1(K)=0$ is clear. The fact that $K$ is norm compact if and only if $\Sigma_1(K)=0$ follows from the fact that $K$ is norm compact if and only if for any bounded, weakly null net $(x_\lambda)\subset B_X$, $\lim_\lambda r_K(x_\lambda)=0$.

$(ii)$ This follows from the fact that $r_K\leqslant R\|\cdot\|$, so $\Sigma_n(K)\leqslant n \Sigma_1(K)\leqslant Rn$.

$(iii)$ Fix $a>\Sigma_n(K)$ and $b>\Sigma_n(L)$.   Fix a weakly null $(x_t)_{t\in D^{\leqslant n}}\subset B_X$.  By applying Proposition \ref{tree} twice, we may fix a pruning $\phi:D^{\leqslant n}\to D^{\leqslant n}$ such that either $r_K(\sum_{i=1}^n x_{\phi(t)|_i}) \leqslant a$ or $r_K(\sum_{i=1}^n x_{\phi(t)|_i}) >a$ for all $t\in D^n$, and such that either $r_L(\sum_{i=1}^n x_{\phi(t)|_i})\leqslant b$ or $r_L(\sum_{i=1}^n x_{\phi(t)|_i})>b$ for all $t\in D^n$.  Since $a>\Sigma_n(K)$, $r_K(\sum_{i=1}^n x_{\phi(t)|_i}) \leqslant a$ for all $t\in D^n$.  Similarly, since $b>\Sigma_n(L)$, $r_L(\sum_{i=1}^n x_{\phi(t)|_i})\leqslant b$ for all $t\in D^n$.    Then fix any $t\in D^n$ and note that $$r_{K+L}(\sum_{i=1}^n x_{\phi(t)|_i})= r_K(\sum_{i=1}^n x_{\phi(t)|_i})+r_L(\sum_{i=1}^n x_{\phi(t)|_i})\leqslant a+b.$$  From this it follows that $$\inf_{t\in D^n} r_{K+L}(\sum_{i=1}^n x_{t|_i}) \leqslant \Sigma_n(K)+\Sigma_n(L).$$   Since $(x_t)_{t\in D^{\leqslant n}}\subset B_X$ was an arbitrary weakly null collection, $\Sigma_n(K+L)\leqslant \Sigma_n(K)+\Sigma_n(L)$. 

$(iv)$ This follows from the fact that $r_{\ee K}(\sum_{i=1}^n x_{t|_i})= r_K(\sum_{i=1}^n \ee x_{t|_i})$ and $(x_t)_{t\in D^{\leqslant n}}\subset B_X$ is weakly null if and only if $(\ee x_t)_{t\in D^{\leqslant n}}\subset B_X$ is.

$(v)$ Obviously $\Sigma_n(\cup_{i=1}^l K_i)\geqslant \max_{1\leqslant i\leqslant l} \Sigma_n(K_i)$.  Now fix $a<\Sigma_n(\cup_{i=1}^l K_i)$ and a weakly null collection $(x_t)_{t\in D^{\leqslant n}}\subset B_X$ such that $$\inf_{t\in D^n} r_{\cup_{i=1}^l K_i}(\sum_{i=1}^n x_{t|_i})>a.$$   Now for each $t\in D^n$, fix $i_t\in \{1, \ldots, l\}$ and $x^*_t\in K_{i_t}$ such that $$x^*_t(\sum_{i=1}^n x_{t|_i})= r_{\cup_{i=1}^l K_i}(\sum_{i=1}^n x_{t|_i}).$$ Define $f:D^n\to \{1, \ldots, l\}$ by $f(t)=i_t$ and fix a pruning $\phi:D^{\leqslant n}\to D^{\leqslant n}$  and $i\in \{1, \ldots, l\}$ such that $f\circ \phi|_{D^n}\equiv i$. We may do this by Proposition \ref{tree}. Then $$a<\inf_{t\in D^n} \text{Re\ }x^*_{\phi(t)}(\sum_{i=1}^n x_{\phi(t)|_i}) \leqslant \Sigma_n(K_i).$$

$(vi)$ Obviously $\Sigma_n(K)\leqslant \Sigma_n(\mathbb{S}K)$. For any $\delta>0$, we may fix a finite subset $T$ of $\mathbb{S}$ such that $\mathbb{S}K \subset (\cup_{\ee\in T} \ee K)+\delta B_{X^*}$. We now combine $(ii)$-$(v)$ to deduce that \begin{align*} \Sigma_n(\mathbb{S}K) & \leqslant \Sigma_n(\cup_{\ee\in T} \ee K)+\Sigma_n(\delta B_{X^*}) \leqslant \Sigma_n(K)+\delta n. \end{align*} Since this holds for any $\delta>0$, we deduce $(vi)$.

$(vii)$ Since $r_{\mathbb{S}K}=r_{\overline{\text{co}}^{\text{weak}^*}(\mathbb{S}K)}$, $$\Sigma_n(\mathbb{S}K)= \Sigma_n(\overline{\text{co}}^{\text{weak}^*}(\mathbb{S}K))= \Sigma_n(\overline{\text{abs\ co}}^{\text{weak}^*}(K)).$$  By $(vi)$, $\Sigma_n(K)=\Sigma_n(\mathbb{S}K)$.

$(viii)$ Assume $\Sigma_n(K)<s'<s$.  Fix $R>0$ such that $K\subset RB_{X^*}$ and $\delta>0$ such that $R\delta n+s'<s$.   Fix a finite $\delta$-net $F$ of $B_{\ell_\infty^n}$ and $(x_t)_{t\in D^{\leqslant n}}\subset B_X$.  By applying Proposition \ref{tree} repeatedly, once for each $(a_i)_{i=1}^n\in F$, we may fix a pruning $\phi:D^{\leqslant n}\to D^{\leqslant n}$ such that for each $(a_i)_{i=1}^n\in F$, either $$r_K(\sum_{i=1}^n a_i x_{\phi(t)|_i})\leqslant s'$$ or $$r_K(\sum_{i=1}^n a_i x_{\phi(t)|_i})>s'$$ for all $t\in D^n$. Since $(a_{|t|}x_t)_{t\in D^{\leqslant n}}\subset B_X$ is weakly null, the latter is impossible.  By our choice of $R$ and $\delta$, we deduce that $$r_K(\sum_{i=1}^n a_i x_{\phi(t)|_i}) \leqslant s$$ for all $(a_i)_{i=1}^n\in B_{\ell_\infty^n}$ and $t\in D^n$. 

The converse is clear.

$(ix)$ Assume $\dim X=\infty$. Let $\Sigma'_n(K)$ be the infimum of those $s>0$ such that for every directed set $D$ and every weakly null $(x_t)_{t\in D^{\leqslant n}}\subset S_X$, $\inf_{t\in D^n} r_K(\sum_{i=1}^n x_{t|_i})\leqslant s$.   It is clear that $\Sigma_n'(K)\leqslant \Sigma_n(K)$.  Seeking a contradiction, assume $s,s'>0$ are such that $\Sigma'_n(K) <s'<s<\Sigma_n(K)$.  Fix $R>0$ such that $K\subset RB_{X^*}$ and fix $\delta>0$ such that $2Rn\delta<s-s'$. Fix $(x_t)_{t\in D^{\leqslant n}}\subset B_X$ such that $$\inf_{t\in D^n} r_K(\sum_{i=1}^n x_{t|_i})>s.$$   By applying Proposition \ref{tree} and relabeling, we may assume there exist  numbers $a_1, \ldots, a_n\in [0,1]$ such that for each $t\in D^{\leqslant n}$, $|\|x_t\|-a_{|t|}|<\delta/2$. Let $I=\{i\leqslant n: a_i\geqslant \delta\}$ and note that $$\inf_{t\in D^n}r_K(\sum_{i\in I} x_{t|_i})>s-R\delta n.$$  Let $M$ be a weak neighborhood basis at $0$ in $X$ and note that there exists a map $\phi:M^{\leqslant |I|}\to D^{\leqslant n}$ such that $(x_{\phi(t)}/\|x_{\phi(t)}\|)_{t\in M^{\leqslant |I|}}$ is weakly null (see \cite[Proposition $7.2$]{C5}). Note that $\Sigma_{|I|}(K)\leqslant \Sigma_n(K)<s'$, since $\dim X=\infty$.    Then with $x'_t=x_{\phi(t)}/\|x_{\phi(t)}\|$, applying Proposition \ref{tree} as usual to $(\ee_i x'_t)_{t\in M^{\leqslant {|I|}}}$ for each $(\ee_i)_{i=1}^{|I|}\in \{\pm 1\}^{|I|}$, we may relabel one more time and assume that for each $t\in M^{|I|}$ and $(\ee_i)_{i=1}^{|I|}\in \{\pm 1\}^{|I|}$, $$r_K(\sum_{i=1}^{|I|} a_i x'_{t|_i})>s-2R\delta n\text{\ \ \ and\ \ \ }r_K(\sum_{i=1} \ee_i x'_{t|_i})<s'.$$   But these conditions are in contradiction, since $r_K$ is sublinear, $s'<s-2R\delta n$, and $\sum_{i=1}^{|I|} a_i x_{t|_i}'$ lies in the convex hull of $\{\sum_{i=1}^{|I|} \ee_i x_{t|_i}': (\ee_i)_{i=1}^{|I|}\in \{\pm 1\}^{|I|}\}$.

\end{proof}

The following uses standard techniques.  It is a generalization of results from \cite{GKL} to arbitrary, weak$^*$-compact sets is the duals of possibly non-separable spaces. We note that these techniques for arbitrary weak$^*$-compact sets and non-separable spaces have appeared for example in \cite[Theorem $2.2$]{C3}. For these reasons,  we only sketch the proof.

\begin{lemma} Let $X$ be a Banach space and let $K\subset X^*$ be weak$^*$-compact. \begin{enumerate}[(i)]\item If $K$ has $M$-summable Szlenk index, $\Sigma(K)\leqslant M$. \item If $\Sigma(K)\leqslant M/4$, then $K$ has $M$-summable Szlenk index. \end{enumerate}

\label{tonytodd}
\end{lemma}

\begin{proof} $(i)$ Assume $\Sigma(K)>M'>M$ and fix $n\in \nn$,  $(x_t)_{t\in D^{\leqslant n}}\subset B_X$ weakly null, and $(x^*_t)_{t\in D^n}\subset K$ such that $$M'<\inf_{t\in D^n} \text{Re\ }x^*_t(\sum_{i=1}^n x_{t|_i}).$$  Fix $R>0$ such that $K\subset RB_{X^*}$ and define $f:D^n\to RB_{(\ell_\infty^n)^\rr}$ by $f(t)=(\text{Re\ }x^*_t(x_{t|_i}))_{i=1}^n$.   Fix $\delta>0$ such that $M+3\delta n<M'$ and apply Proposition \ref{tree} and relabel to assume there exists a sequence $(a_i)_{i=1}^n \in RB_{(\ell_\infty^n)^\rr}$ such that $$|a_i-\text{Re\ }x^*_t(x_{t|_i})|<\delta$$ for all $t\in D^n$ and $1\leqslant i\leqslant n$.  Then $$M'<n\delta+\sum_{i=1}^n a_i.$$  Now an easy induction proof yields that for any $0\leqslant i\leqslant n$ and any $t\in \{\varnothing\}\cup D^{\leqslant i}$,  there exists $x^*_t\in s_{a_{i+1}-2\delta}\ldots s_{a_n-2\delta}(K)$ such that if $\varnothing<s\leqslant t$, $\text{Re\ }x^*_t(x_s)\geqslant a_{|s|}-\delta$. In particuclar, $x^*_\varnothing\in s_{a_1-2\delta}\ldots s_{a_n-2\delta}(K)$.  Since $$\sum_{i=1}^n (a_i-2\delta) > M'-3\delta n >M,$$ this shows that $K$ does not have $M$-summable Szlenk index.

$(ii)$  Assume that $K$ does not have $M$-summable Szlenk index.  Then there exist $\ee_1, \ldots, \ee_n>0$ such that $s_{\ee_1}\ldots s_{\ee_n}(K)\neq \varnothing$ and $\sum_{i=1}^n \ee_i=M'>M$. Fix $\delta>0$ such that $M'-\delta n>M$.     Let $D$ be a weak neighborhood basis at $0$ in $X$ and let $N$ be a weak$^*$-neighborhood basis at $0$ in $X^*$.  Then by standard techniques, we may fix $(x^*_t)_{t\in \{\varnothing\}\cup N^{\leqslant n}}\subset K$ such that for each $t\in \{\varnothing\}\cup D^{\leqslant n-1}$, $\text{weak}^*$-$\lim_{v\in N}x^*_{t\smallfrown v}=x^*_t$ and for each $v\in N$, $\|x^*_{t\smallfrown v}-x^*_t\|>\ee_{|t|+1}/2$.  Now we may define a map $\phi:D^{\leqslant n}\to N^{\leqslant n}$ and a weakly null collection $(x_t)_{t\in D^{\leqslant n}}\subset B_X$ such that $\text{Re\ }x^*_{\phi(t)}(x_s)\geqslant (\ee_{|s|}-\delta)/4$ for any $\varnothing<s\leqslant t$.  In particular, $$\inf_{t\in D^n} r_K(\sum_{i=1}^n x_{t|_i}) \geqslant \inf_{t\in D^n} \text{Re\ }x_{\phi(t)}(\sum_{i=1}^n x_{t|_i}) \geqslant \frac{1}{4}\Bigl(\sum_{i=1}^n \ee_i - n\delta\Bigr)>M/4.$$  This shows that $\Sigma(K)>M/4$.

\end{proof}

\begin{corollary} Let $X$ be a Banach space and let $K\subset X^*$ be weak$^*$-compact.  Then $K$ has summable Szlenk index if and only if $\Sigma(K)<\infty$ if and only if $\overline{\text{\emph{abs\ co}}}^{\text{\emph{weak}}^*}(K)$ has summable Szlenk index. 

\label{main corollary}

\end{corollary}

For each operator $A$, let $\mathfrak{s}(A)=\|A\|+\Sigma(A)$ and let $\mathfrak{S}$ denote the class of all operators with summable Szlenk index. Note that by Corollary \ref{main corollary}, $\mathfrak{S}$ is the class of all operators $A$ such that $\mathfrak{s}(A)<\infty$.

\begin{theorem} The class $(\mathfrak{S}, \mathfrak{s})$ is a Banach ideal.

\label{ideal}

\end{theorem}

\begin{proof} Fix $X,Y\in \textbf{Ban}$ and note that by Proposition \ref{tech1} and the positive homogeneity of $\Sigma$, $\Sigma$ is a seminorm on $\mathfrak{S}(X,Y)$. From this we can deduce that $(\mathfrak{S}(X,Y), \mathfrak{s})$ is a normed space.  

Now fix $W,Z\in \textbf{Ban}$, $C:W\to X$, $B:X\to Y$, and $A:Y\to Z$ with $\|A\|=\|C\|=1$.  Fix $n\in \nn$ and a weakly null $(x_t)_{t\in D^{\leqslant n}}\subset B_W$.   Then $(Cw_t)_{t\in D^{\leqslant n}}\subset B_X$ is weakly null, and $$\inf_{t\in D^n}\|ABC\sum_{i=1}^n w_{t|_i}\| \leqslant \inf_{t\in D^n} \|B\sum_{i=1}^n Cw_{t|_i}\|\leqslant \Sigma_n(B).$$    Thus $\Sigma_n(ABC)\leqslant \Sigma_n(B)$.  By homogeneity, we deduce that $\Sigma_n(ABC)\leqslant \|A\|\Sigma_n(B)\|C\|$ and $\mathfrak{s}(ABC)\leqslant \|A\|\mathfrak{s}(B)\|C\|$ for any $C:W\to X$ and $A:Y\to Z$. 

Next, since $\Sigma(A)=0$ for any compact operator, $\mathfrak{S}$ contains all finite rank operators and $\mathfrak{s}(x\otimes y)= \|x\otimes y\|=\|x\|\|y\|$ for each $x\in X$ and $y\in Y$. 

It remains to show that $(\mathfrak{S}(X,Y), \mathfrak{s})$ is complete.  To that end, fix a $\mathfrak{s}$-Cauchy sequence $(A_k)_{k=1}^\infty$ in $\mathfrak{S}(X,Y)$.  Since $(A_k)_{k=1}^\infty$ is also norm Cauchy, it is norm convergent to some $A$.  Since $\Sigma_n(A-A_k)\leqslant n \|A-A_k\|$ for any $n,k\in \nn$, it follows that $$\Sigma(A)=\underset{n}{\ \sup\ } \Sigma_n(A) \leqslant \underset{n}{\ \sup\ } \underset{k}{\ \lim\sup\ } \Sigma_n(A_k) \leqslant \underset{k}{\ \lim\sup\ } \Sigma(A_k)<\infty$$ and $$\underset{n}{\ \lim\sup\ } \Sigma(A-A_n) \leqslant \underset{n}{\ \lim\sup\ }\underset{k}{\ \lim\sup\ } \Sigma(A_k-A_n)=0.$$

\end{proof}

\begin{rem}\upshape The class $\mathfrak{S}$ is not a closed ideal. Indeed, let $X_n$ be the completion of $c_{00}$ with respect to the norm $$\|\sum_{i=1}^\infty a_i e_i\|_{X_n}= \max\Bigl\{\sum_{i\in T} |a_i|: |T|=n\Bigr\}.$$  It is quite clear that $\Sigma(X_n)=n$, so that $A:(\oplus_{n=1}^\infty X_n)_{c_0}\to (\oplus_{n=1}^\infty X_n)_{c_0}$ given by $A|_{X_n}=n^{-1/2}I_{X_n}$ quite obviously fails to have summable Szlenk index, but is the norm limit of operators which have summable Szlenk index.

\end{rem}

\section{Embedding}

The equivalence of $(i)$ and $(iii)$ of the next theorem is no doubt known to specialists.  We are unaware of any mention of this fact in the literature, and we will need it for later results, so we include it here.

\begin{theorem} Let $A:X\to Y$ be an operator.   The following are equivalent.  \begin{enumerate}[(i)]\item $\Sigma(A)<\infty$. \item $A$ has summable Szlenk index.

\vspace{5mm}

\noindent Furthermore, if $A=I_X$ and $\dim X=\infty$,  each of the above is equivalent to \item $X$ is Asymptotic $c_0$.

\vspace{5mm}

\noindent Finally, if $A=I_X$, $\dim X=\infty$,  and $X$ has a shrinking FDD $E$, each of the above is equivalent to \item There exists a blocking $F$ of $E$ which is asymptotic $c_0$ in $X$. 

\end{enumerate} 

\label{main thm}

\end{theorem}

\begin{proof}[Proof of Theorem \ref{main thm}] The equivalence of $(i)$ and $(ii)$ comes from Corollary \ref{main corollary}.  The equivalence $(ii)\Rightarrow (iii)$ follows from Proposition \ref{tech1}$(viii)$ and $(ix)$.

Assume $A=I_X$, $\dim X=\infty$, and $E$ is a shrinking FDD for $X$. Fix $C\geqslant 1$ such that $|\sum_{i=1}^n e_i|\leqslant C$ for every $n\in \nn$ and $|\cdot|\in \{X\}_n$.   Fix $C_1>C$.  For an infinite subset $M$ of $\nn$, if $M=\{m_1, m_2, \ldots\}$ with $m_1<m_2<\ldots$ and $m_0=0$, let $F^M$ is the blocking of $E$ such that $F^M_n=\oplus_{j=m_{n-1}+1}^{m_n}F_j$.   Let $\mathcal{V}$ denote the set of those infinite subsets $M$ of $\nn$ such that there exists $(x_i)_{i=1}^n\in B_X\cap \prod_{i=1}^n F^M_{2i}$ with $\|\sum_{i=1}^n x_i\|\geqslant C_1$. Arguing as in \cite[Theorem $3.3$]{OS}, we deduce the existence of some infinite subset $M$ of $\nn$ such that for any infinite subset $N$ of $M$, $N\notin \mathcal{V}$. From the definition of $\mathcal{V}$, if $(x_i)_{i=1}^{2n}\subset B_X$ is any block sequence of $F^M$, then $$\|\sum_{i=1}^{2n} x_i\|\leqslant \|\sum_{i=1}^n x_{2i-1}\|+\|\sum_{i=1}^n x_{2i}\|\leqslant 2C_1.$$  Since we may do this for any $n$, a standard diagonalization procedure yields that $(iii)\Rightarrow (iv)$. 

 Last, $(iv)\Rightarrow (iii)$ is obvious.

\end{proof}

The following result provides a negative solution to  a conjecture from \cite{GKL}.

\begin{corollary} There exists an $\ell_1$ predual which has summable Szlenk index but contains no isomorph of $c_0$. 

\end{corollary}

\begin{proof} By \cite[Proposition $5.7$]{AGM}, there exists a $\mathcal{L}_\infty$ Banach space $X$ with FDD $E$ such that $E$ is asymptotic $c_0$ in $X$ and such that $X$ contains no isomorph of $c_0$ and $X^*$ is isomorphic to $\ell_1$. This space $X$ has summable Szlenk index.

\end{proof}

The following result generalizes a theorem from \cite{OSZ}, where it was shown that any separable, reflexive, Asymptotic $c_0$ space embeds into a Banach space with $Z$ with FDD $E$ such that $E$ is asymptotic $c_0$ in $Z$. 

\begin{theorem} Let $X$ be a separable Banach space.  Then $X$ is Asymptotic $c_0$ if and only if there exists a Banach space $Z$ with FDD $E$ such that $E$ is asymptotic $c_0$ in $Z$ and $Z$ is isometric to a subspace of $Z$. Moreover, if $X$ is reflexive, $Z$ can be taken to be reflexive. 

\end{theorem}

\begin{proof} By \cite{Sch}, there exists a weak$^*$-compact set $B\subset B_{X^*}$ and a Banach space $Z$ with shrinking FDD $E$ such that $X$ embeds isomorphically into $Z$ and such that $Z$ is reflexive if $X$ is. Furthermore, there exist a subset $\mathbb{B}\subset B_{Z^*}$ such that $\overline{\text{abs\ co}}^{\text{weak}^*}(\mathbb{B})=B_{Z^*}$, a constant $c>0$, and a map $I^*:Z^*\to X^*$ such that $$I^*(s_{\ee_1}\ldots s_{\ee_n}(\mathbb{B}))\subset s_{\ee_1/c}\ldots s_{\ee_n/c}(B).$$  Each of these properties except the last comes from the construction of the space $Z$. The last property follows from an inessential modification of \cite[Lemma $5.5$]{Sch}. If $X$ has summable Szlenk index, so does $B$, and therefore so does $\mathbb{B}$. By Corollary \ref{main corollary}, $B_{Z^*}=\overline{\text{abs\ co}}^{\text{weak}^*}(\mathbb{B})$ has summable Szlenk index as well.  This means $Z$ is Asymptotic $c_0$, and therefore some blocking of $E$ is asymptotic $c_0$ in $Z$.

\end{proof}

\section{Injective tensor products}

Let us recall that the injective tensor product is the closed span  in $\mathfrak{L}(Y^*, X)$ of the operators $x\otimes y:Y^*\to X$, where $x\otimes y(y^*)= y^*(y)x$. For $i=0,1$, if $A_i:X_i\to Y_i$ is an operator, we may define the operator $A_0\otimes A_1:X_0\hat{\otimes}_\ee X_1\to Y_0\hat{\otimes}_\ee Y_1$. This operator is given by $A_0\otimes A_1 (u)= A_0 u A_1^*:Y^*_1\to Y_0$.  Given subsets $K_0\subset X^*_0$, $K_1\subset X^*_1$, we let $$[K_0, K_1]=\{x^*_0\otimes x^*_1: x^*_0\in K_0, x^*_1\in K_1\}\subset (X_0\hat{\otimes}_\ee X_1)^*.$$

\begin{proposition} Let $J$ be a finite set.  Suppose that $R>0$ and for each $i=0,1$ and $j\in J$, $K_{i,j}\subset R B_{X^*_i}$ is a weak$^*$-compact set. Then for any $\ee_1, \ldots, \ee_n\in \rr$ and any $n\in \{0\}\cup \nn$, $$s_{\ee_1}\ldots s_{\ee_n}\Bigl(\bigcup_{j\in J}[K_{0,j}, K_{1,j}]\Bigr)\subset \bigcup_{j\in J,(k_i)_{i=1}^n\in \{0,1\}^n} [s_{\ee_1/4R}^{k_1}\ldots s_{\ee_n/4R}^{k_n}(K_{0,j}), s_{\ee_1/4R}^{1-k_1}\ldots s_{\ee_n/4R}^{1-k_n}(K_{1,j})].$$   

\label{tensor}
\end{proposition}

\begin{proof} We induct on $n$ with the $n=0$ case true by definition.

It is easy to see that if $R>0$, $x^*_0, z^*_0\in RB_{X^*_0}$, $x^*_1, z^*_1\in RB_{X^*_1}$,  and $\|x^*_0\otimes x^*_1-z^*_0\otimes z^*_1\|>\ee$, then $$\max\{\|x^*_0-z^*_0\|, \|x^*_1-z^*_1\|\} > \ee/2R.$$  Now assume the result holds for $n$ and $$u^*\in s_{\ee_1}\ldots s_{\ee_{n+1}}\Bigl(\cup_{j\in J}[K_{0,j}, K_{1,j}]\Bigr)=s_{\ee_1}\Biggl(s_{\ee_2}\ldots s_{\ee_{n+1}}\Bigl(\cup_{j\in J}[K_{0,j}, K_{1,j}]\Bigr)\Biggr).$$ This means there exists a net $(u^*_\lambda)\subset s_{\ee_2}\ldots s_{\ee_{n+1}}\Bigl(\cup_{j\in J}[K_{0,j}, K_{1,j}]\Bigr)$ converging weak$^*$ to $u^*$ such that $\|u^*-u^*_\lambda\|>\ee_1/2$ for all $\lambda$.  By the inductive hypothesis, for each $\lambda$ there exists $j_\lambda\in J$ and $(k^\lambda_i)_{i=2}^{n+1}\in \{0,1\}^n$ such that $$u^*_\lambda\in [s_{\ee_2/4R}^{k_2^\lambda}\ldots s_{\ee_{n+1}/4R}^{k^\lambda_{n+1}}(K_{0, j_\lambda}), s_{\ee_2/4R}^{1-k_2^\lambda}\ldots s_{\ee_{n+1}/4R}^{1-k^\lambda_{n+1}}(K_{1, j_\lambda})].$$   By passing to a subnet, we may assume there exist $j\in J$ and $(k_i)_{i=2}^{n+1}\in \{0,1\}^n$ such that $j=j_\lambda$ for all $\lambda$, $k_i=k^\lambda_i$ for all $\lambda$ and $2\leqslant i\leqslant n+1$. For each $\lambda$, write $$u^*_\lambda= x^*_{0, \lambda}\otimes x^*_{1, \lambda}\in [s_{\ee_2/4R}^{k_2}\ldots s_{\ee_{n+1}/4R}^{k_{n+1}}(K_{0, j}), s_{\ee_2/4R}^{1-k_2}\ldots s_{\ee_{n+1}/4R}^{1-k_{n+1}}(K_{1, j})].$$ By passing to a subnet again, we may assume $x^*_{0, \lambda}\underset{\text{weak}^*}{\to}x^*_0\in s_{\ee_2/4R}^{k_2}\ldots s_{\ee_{n+1}/4R}^{k_{n+1}}(K_{0, j})$, $x^*_{1, \lambda}\underset{\text{weak}^*}{\to} x^*_1\in s_{\ee_2/4R}^{1-k_2}\ldots s_{\ee_{n+1}/4R}^{1-k_{n+1}}(K_{1, j})$, and  either $$\|x^*_0-x^*_{0, \lambda}\|> \ee_1/4R$$ for all $\lambda$ or $$\|x^*_1-x^*_{1, \lambda}\|>\ee_1/4R$$ for all $\lambda$. For this we are using the fact that $u^*=x^*_0\otimes x^*_1$. If $\|x^*_0-x^*_{0, \lambda}\|>\ee_1/4R$ for all $\lambda$, let $k_1=1$, and otherwise let $k_1=0$.   Then $$u^*= x^*_0\otimes x^*_1\in [s^{k_1}_{\ee_1/4R}\ldots s^{k_{n+1}}_{\ee_{n+1}/4R}(K_{0,j}), s^{1-k_1}_{\ee_1/4R}\ldots s^{1-k_{n+1}}_{\ee_{n+1}/4R}(K_{1, j})].$$

\end{proof}

\begin{corollary} Let $A_0:X_0\to Y_0$, $A_1:X_1\to Y_1$ be non-zero operators.  Then $A_0, A_1$ have summable Szlenk index if and only if $A_0, A_1$ do.  

\end{corollary}

\begin{proof}  If $A_0\otimes A_1$ has summable Szlenk index, by the ideal property, $A_0, A_1$ do.

Let $K=[A^*_0B_{Y^*_0}, A^*_1 B_{Y^*_1}]\subset (X_0\hat{\otimes}_\ee X_1)^*$ and note that  $\overline{\text{abs\ co}}^{\text{weak}^*}(K)=(A_0\otimes A_1)^* B_{(Y_0\hat{\otimes}_\ee Y_1)^*}$ by the Hahn-Banach theorem.    By Corollary \ref{main corollary}, it is sufficient to show that $K$ has summable Szlenk index.  Assume $A_0$ has $M_0$-summable Szlenk index and $A_1$ has $M_1$-summable Szlenk index. Let $R=\max\{\|A_0\|, \|A_1\|\}$.  Fix $\ee_1, \ldots, \ee_n>0$ such that $\sum_{i=1}^n \ee_i >4R(M_0+M_1)$.  Then for any $(k_i)_{i=1}^n\in \{0,1\}^n$, $$M_0+M_1< \sum_{i=1}^n k_i \ee_i/4R+\sum_{i=1}^n (1-k_i)\ee_i/4R,$$ so that either $\sum_{i=1}^n k_i \ee_i/4R>M_0$ or $\sum_{i=1}^n (1-k_i)\ee_i/4R>M_1$. In either case, $$[s_{\ee_1/4R}^{k_1}\ldots s_{\ee_n/4R}^{k_n}(A^*_0 B_{Y^*_0}), s_{\ee_1/4R}^{1-k_1}\ldots s_{\ee_n/4R}^{1-k_n}(A^*_1 B_{Y^*_1})]=\varnothing.$$   By Proposition \ref{tensor}, $$s_{\ee_1}\ldots s_{\ee_n}(K)\subset \bigcup_{(k_i)_{i=1}^n\in\{0,1\}^n} [s_{\ee_1/4R}^{k_1}\ldots s_{\ee_n/4R}^{k_n}(A^*_0 B_{Y^*_0}), s_{\ee_1/4R}^{1-k_1}\ldots s_{\ee_n/4R}^{1-k_n}(A^*_1 B_{Y^*_1})]=\varnothing,$$ whence $K$ has $4R(M_0+M_1)$-summable Szlenk index.

\end{proof}

The following answers a question from \cite{DK}. 

\begin{corollary} Let $X_0, X_1$ be non-zero Banach spaces.  Then $X_0\hat{\otimes}_\ee X_1$ is Asymptotic $c_0$  if and only if $X_0, X_1$ are. Equivalently, $X_0\hat{\otimes}_\ee X_1$ has summable Szlenk index if and only if $X_0, X_1$ do.

\end{corollary}

\section{Direct sums}

The first result of this section is an operator version of a result from \cite{DK2}. However, we will use Corollary \ref{main corollary} to give a new proof.

\begin{theorem} Suppose that $\Lambda$ is a non-empty set and for each $\lambda\in \Lambda$, $A_\lambda:X_\lambda\to Y_\lambda$ is an operator.  Assume also that $\sup_{\lambda\in \Lambda}\|A_\lambda\|<\infty$ and let $A:(\oplus_{\lambda\in \Lambda}X_\lambda)_{c_0(\Lambda)}\to (\oplus_{\lambda\in \Lambda} Y_\lambda)_{c_0(\Lambda)}$ be the operator such that $A|_{X_\lambda}=A_\lambda$.

Then $A$ has summable Szlenk index if and only if there exists $M$ such that for each $\lambda\in \Lambda$, $A_\lambda$ has $M$-summable Szlenk index. 
\label{sum}
\end{theorem}

\begin{proof}  It is clear that if $A$ has $M$-summable Szlenk index, $A_\lambda$ has $M$-summable Szlenk index for each $\lambda\in \Lambda$, which gives one direction.

Now suppose there exists $M$ such that $A_\lambda$ has $M$-summable Szlenk index for each $\lambda\in \Lambda$.  Let $K= \bigcup_{\lambda\in \Lambda} A^*_\lambda B_{Y^*_\lambda}$.  It is clear that for any $n\in \nn$ and $\ee_1, \ldots, \ee_n$, $$s_{\ee_1}\ldots s_{\ee_n}(K) \subset \{0\} \cup \bigcup_{\lambda\in \Lambda} s_{\ee_1}\ldots s_{\ee_n}(A^*_\lambda B_{X_\lambda^*}).$$  From this it follows that with $M'=M+ 2 \sup_{\lambda\in \Lambda} \|A_\lambda\|$, $K$ has $M'$-summable Szlenk index.  Indeed, suppose $\ee_1, \ldots, \ee_n>0$ are such that $\sum_{i=1}^n \ee_i>M'$.   Note that  $\ee_i \leqslant 2\sup_{\lambda\in \Lambda} \|A_\lambda\|$ for each $1\leqslant i\leqslant n$.  This means that $\sum_{i=2}^n \ee_i >M$, whence $$s_{\ee_1}\ldots s_{\ee_n}(K) \subset s_{\ee_1}\Bigl( \{0\}\cup\bigcup_{\lambda\in \Lambda} s_{\ee_2}\ldots s_{\ee_n}(A^*_\lambda B_{Y^*_\lambda})\Bigr) = s_{\ee_1}(\{0\})=\varnothing.$$

\end{proof}

We next turn to a facet of this problem which is of interest for operators, but not for spaces.  Above we considered $c_0$ direct sums, while below we wish to consider $\ell_p$ direct sums, $1\leqslant p\leqslant \infty$. However, if $(X_\lambda)_{\lambda\in \Lambda}$ is a collection of non-zero Banach spaces, $(\oplus_{\lambda\in \Lambda} X_\lambda)_{\ell_p(\Lambda)}$  contains a copy of $\ell_p$ and therefore cannot have summable Szlenk index  except in the case that $\Lambda$ is finite. Our final goal is to elucidate the situation for operators.

\begin{proposition} Fix $1\leqslant p\leqslant \infty$.   For any operators $A_i:X_i\to Y_i$, $1\leqslant i\leqslant k$ and $n\in \nn$, $$\Sigma_n\Bigl(A:(\oplus_{i=1}^k X_i)_{\ell_p^k}\to (\oplus_{i=1}^k Y_i)_{\ell_p^k}\Bigr) \leqslant \|(\Sigma_n(A_i))_{i=1}^k\|_{\ell_p^k}$$  and $$\Sigma\Bigl(A:(\oplus_{i=1}^k X_i)_{\ell_p^k}\to (\oplus_{i=1}^k Y_i)_{\ell_p^k}\Bigr) = \|(\Sigma(A_i))_{i=1}^k\|_{\ell_p^k}.$$

\label{finite}
\end{proposition}

\begin{proof} In the proof, we identify $X_i$ and $Y_i$ with subspaces of $X=(\oplus_{i=1}^k X_i)_{\ell_p^k}$ and $Y=(\oplus_{i=1}^k Y_i)_{\ell_p^k}$, respectively. Let $A:X\to Y$ denote the operator with $A|_{X_j}=A_j$.  Let $P_j:X\to X$ denote the projection from $X$ onto $X_j$. Then  $\Sigma_n(A_i)=\Sigma_n(AP_i)$.

Fix $n\in \nn$ and a weakly null collection $(x_t)_{t\in D^{\leqslant n}}\subset B_X$.    For each $i\in I$, fix $a_i>\Sigma_n(A_i)$.  By applying Proposition \ref{tree} to $(P_jx_t)_{t\in D^{\leqslant n}}$ for each $j=1, \ldots, k$ and relabeling, we may assume $$\|AP_j\sum_{m=1}^n x_{t|_m} \|\leqslant a_j,$$  and $$\|A\sum_{m=1}^n x_{t|_m}\|=\|(\|AP_j\sum_{m=1}^n x_{t|_m}\|)_{j=1}^k\|_{\ell_p^k}\leqslant \|(a_i)_{i=1}^k\|_{\ell_p^k}.$$  Since $a_i>\Sigma_n(A_i)$ was arbitrary, we conclude $$\Sigma_n(A) \leqslant \|(\Sigma_n(A_i))_{i=1}^k\|_{\ell_p^k}.$$

Now for each $1\leqslant i\leqslant k$, fix $0<b_i<\Sigma(A_i)$ if $\Sigma(A_i)>0$ and otherwise let $b_i=0$. If $b_i>0$, fix  $n_i\in \nn$ such that $b_i<\Sigma_{n_i}(K_i)$, and otherwise let $n_i=1$.  Let $n=\max_{1\leqslant i\leqslant k}n_i$. Let $D$ be a weak neighborhood basis at $0$ in $X$.  By Remark \ref{remrem}, we may fix for each $1\leqslant j\leqslant k$ some weakly null collection $(x^j_t)_{t\in D^{\leqslant n}}\subset B_{X_j}\subset B_X$ such that $$\inf_{t\in D^n}\|A\sum_{m=1}^n x^j_{t|_m}) \geqslant b_j.$$   Now fix $t\in D^{\leqslant kn}$ and assume $|t|=(j-1)n+r$, $j,r\in  \nn$, $0\leqslant r<n$. We may write $t=s\smallfrown s'$, where $|s|=(j-1)n$ and $|s'|=r$, and let $x_t=x^j_{s'}$. Then $(x_t)_{t\in D^{\leqslant kn}}\subset B_X$ is weakly null and $$\inf_{t\in D^{kn}} \|A\sum_{m=1}^{kn} x_{t|_m}\| \geqslant \|(b_i)_{i=1}^k\|_{\ell_p^k}.$$   This shows that $\Sigma(A)\geqslant \|(\Sigma(A_i))_{i=1}^k\|_{\ell_p^k}$. The reverse inequality follows from the previous paragraph.

\end{proof}

\begin{corollary} Fix $1\leqslant p\leqslant \infty$.  Assume that $\Lambda$ is a non-empty set and for each $\lambda\in \Lambda$, $A_\lambda:X_\lambda\to Y_\lambda$ is an operator.  Assume also that $\sup_{\lambda\in \Lambda} \|A_\lambda\|<\infty$ and let $A:(\oplus_{\lambda\in \Lambda} X_\lambda)_{\ell_p(\Lambda)} \to (\oplus_{\lambda\in \Lambda}Y_\lambda)_{\ell_p(\Lambda)}$ be the operator such that $A|_{X_\lambda}=A_\lambda$.  Then $A$ has summable Szlenk index if and only if $(\|A_\lambda\|)_{\lambda\in \Lambda}\in c_0(\Lambda)$ and $(\Sigma(A_\lambda))_{\lambda\in \Lambda}\in \ell_p(\Lambda)$.   Moreover, in this case, $$\Sigma(A)=\|(\Sigma(A_\lambda))_{\lambda\in \Lambda}\|_{\ell_p(\Lambda)}.$$

\end{corollary}

\begin{proof} Throughout the proof, for a finite subset $\Upsilon$ of $\Lambda$, let $P_\Upsilon A$ denote the map given by $P_\Upsilon A|_{X_\lambda}=A_\lambda$ if $\lambda\in \Upsilon$ and $P_\Upsilon A|_{X_\lambda}=0$ if $\lambda\in \Lambda\setminus \Upsilon$.

If $(\|A_\lambda\|)_{\lambda\in \Lambda}\in \ell_\infty(\Lambda)\setminus c_0(\Lambda)$, then $A$ preserves an isomoprhic copy of $\ell_p$ and cannot have summable Szlenk index. By Proposition \ref{finite},  \begin{align*} \Sigma(A) & \geqslant \sup \{\Sigma(P_\Upsilon A): \Upsilon\subset \Lambda\text{\ finite}\} = \sup\{\|(\Sigma(A_\lambda))_{\lambda\in \Upsilon}\|_{\ell_p(\Upsilon)}: \Upsilon \subset \Lambda\text{\ finite}\} \\ & = \|(\Sigma(A_\lambda))_{\lambda\in \Lambda}\|_{\ell_p(\Lambda)}.\end{align*} Therefore if $A$ has summable Szlenk index, $(\|A_\lambda\|)_{\lambda\in \Lambda}\in c_0(\Lambda)$ and $\|(\Sigma(A_\lambda))_{\lambda\in \Lambda}\|_{\ell_p(\Lambda)}\leqslant \Sigma(A)<\infty$.

Now if $(\|A_\lambda\|)_{\lambda\in \Lambda}\in c_0(\Lambda)$ and $\|(\Sigma(A_\lambda))_{\lambda\in \Lambda}\|_{\ell_p(\Lambda)}<\infty$, $A\in \overline{\{P_\Upsilon A:\Upsilon\subset \Lambda\text{\ finite}\}}$.  Arguing as in the proof of Theorem \ref{ideal}, $$\Sigma(A)\leqslant \sup \{\Sigma(P_\Upsilon A): \Upsilon\subset \Lambda\text{\ finite}\}= \|(\Sigma(A_\lambda))_{\lambda\in \Lambda}\|_{\ell_p(\Lambda)}<\infty.$$

\end{proof}

\end{document}